\documentclass[11pt]{article}
\usepackage[utf8]{inputenc}
\usepackage[T1]{fontenc}
\usepackage{lmodern}
\usepackage{microtype}
\usepackage{amsmath,amssymb,amsthm,mathtools}
\usepackage{graphicx}
\usepackage[a4paper,margin=1in]{geometry}
\usepackage{enumitem}
\usepackage{hyperref}
\hypersetup{colorlinks=true,linkcolor=blue,citecolor=blue,urlcolor=blue}

\newtheorem{theorem}{Theorem}[section]
\newtheorem{lemma}[theorem]{Lemma}
\newtheorem{proposition}[theorem]{Proposition}

\newtheorem{corollary}[theorem]{Corollary}
\theoremstyle{definition}
\newtheorem{definition}[theorem]{Definition}
\newtheorem{example}[theorem]{Example}
\theoremstyle{remark}
\newtheorem{remark}[theorem]{Remark}

\title{Stochastic realisers of non-degenerate full-degree Type II reduced Ito polynomials}
\author{
Brecht Verbeken$^{1,2}$ and Vincent Ginis$^{1,2,3}$\\[0.5em]
$^1$Department of Business Technology and Operations, Data Analytics Laboratory,\\
Vrije Universiteit Brussel (VUB), Pleinlaan 2, 1050 Brussels, Belgium\\
$^2$imec-SMIT, Vrije Universiteit Brussel, Pleinlaan 9, 1050 Brussels, Belgium\\
$^3$School of Engineering and Applied Sciences, Harvard University,\\
Cambridge, Massachusetts 02138, USA\\[0.5em]
\texttt{brecht.verbeken@vub.be}\quad \texttt{vincent.ginis@vub.be}
}
\date{}

\begin{document}
\maketitle

\begin{abstract}
Let
\[
        f_\alpha(x)=\bigl(x^q-(1-\alpha)\bigr)^d-\alpha^d x^z,
        \qquad 0<\alpha<1,
\]
where $q\ge2$, $d\ge2$, $1\le z\le q-1$, and $\gcd(q,z)=1$. These are the non-degenerate full-degree Type~II reduced Ito polynomials of order $n=qd$. We give a complete parametrisation, up to permutation similarity, of the stochastic matrices whose characteristic polynomial is $f_\alpha$.

The support data are a composition $a_0+\cdots+a_{d-1}=z$ and a cyclic origin $\rho\in\mathbb Z_q$. The transfer support in block $0$ is an arbitrary nonempty subset of $[\rho,\rho+a_0]_q$, while the support in block $t\ge1$ is an arbitrary nonempty subset of
\[
\left[
 \rho-\sum_{j=1}^t(a_j+1),\,
 \rho-\sum_{j=1}^t(a_j+1)+a_t
\right]_q .
\]
The horizontal weights at the selected positions are arbitrary elements of $(0,1)$ whose product in every block is $1-\alpha$. The proof first shows that the order-$n$ Type~II arc lies strictly outside $\Theta_{n-1}$, which permits the Dmitriev--Dynkin/Kirkland--\v Smigoc two-shift reduction. A circular separation theorem then classifies the possible supports, and Coates' formula supplies the characteristic polynomial. This resolves the non-sparse full-degree Type~II characteristic-polynomial realiser problem.
\end{abstract}

\noindent\textbf{Keywords:} stochastic matrix, Karpelevich region, reduced Ito polynomial, nonnegative inverse eigenvalue problem, Coates formula, directed graph

\noindent\textbf{2020 MSC:} 15A18, 15B51, 05C20, 60J10

\section{Introduction}

The eigenvalues of finite stochastic matrices are constrained in a way that is both spectral and combinatorial.  If $\Theta_n$ denotes the set of all complex numbers that occur as eigenvalues of $n\times n$ stochastic matrices, then Karpelevich's theorem describes $\Theta_n$ exactly \cite{Karpelevich1951}; see also the modern account in \cite{KirklandLaffeySmigoc2020}.  Ito's formulation expresses the boundary arcs of $\Theta_n$ through one-parameter polynomial equations attached to Farey-neighbour data \cite{Ito1997}.  After removing extraneous zero factors, these equations give the reduced Ito polynomials of Types 0, I, II, and III.

A distinct problem, and the one considered here, is not merely to know where the boundary eigenvalues lie, but to understand the stochastic matrices that realise the corresponding boundary polynomials as characteristic polynomials.  Johnson and Paparella gave a parametric realising matrix for each Karpelevich arc \cite{JohnsonPaparella2017}.  Kirkland and \v Smigoc then initiated a structural realisation programme \cite{KirklandSmigoc2022}.  They proved complete classifications for the Type 0 and Type I reduced Ito polynomials, and for Types II and III they proved the corresponding classifications under a sparsity hypothesis.  Their Type II section makes the remaining obstruction explicit: after the forced $q$-cycles are present, there may be more than one edge between consecutive $q$-cycle blocks, and therefore there may be many $(n-z)$-cycles rather than a single one.  They give a recursive edge-addition discussion and examples, but no closed form for all non-sparse Type II supports.

Joshi, Kirkland and \v Smigoc later studied powers of Karpelevich arcs and powers of their sparsest realising matrices \cite{JoshiKirklandSmigoc2023}.  Their matrix results are formulated for the sparsest class $M_n^0(q,s)$, so they do not subsume the non-sparse Type II support problem solved here.  The companion Type III manuscript of the present authors proves the Kirkland--\v Smigoc Type III realisation conjecture for genuine Type III reduced Ito polynomials in the range $0<\alpha\le 1$ \cite{VerbekenGinisTypeIII}.  That result is also separate: Type III has a global $n$-cycle together with backward $q$-edges, whereas Type II has $d$ horizontal $q$-cycles and transfer edges between the blocks.  The two finite support geometries are different.

This article supplies the missing closed form for Type II.  In the full-degree Type II case one has
\[
        n=qd,\qquad d\ge 2,
\]
and the reduced Ito polynomial is
\begin{equation}
        f_\alpha(x)=\bigl(x^q-\beta\bigr)^d-\alpha^d x^z,
        \qquad \beta=1-\alpha,
        \qquad 1\le z\le q-1 .
        \label{eq:typeIIpoly}
\end{equation}
Put $s=qd-z$.  The determinant-one Farey condition for the endpoints neighbouring $q$ is exactly
\[
        \gcd(q,z)=\gcd(q,qd-z)=1 .
\]
This coprimality is not cosmetic: it is the congruence that prevents a simple transfer cycle from winding more than once through each $q$-cycle block.  Lemma \ref{lem:typeII-exact-order} below proves that every such bare arithmetic datum gives an exact-order full-degree Type II arc: its non-endpoint points lie in $\Theta_{qd}\setminus\Theta_{qd-1}$.

The result is easiest to state in a normal form.  The vertices are
\[
        \mathbb Z_d\times\mathbb Z_q .
\]
The horizontal edges are
\[
        (t,i)\longrightarrow (t,i+1),
\]
and the possible transfer edges are
\[
        (t,i)\longrightarrow (t+1,i),\qquad t=0,\ldots,d-2,
\]
while the last block has possible transfer edges
\[
        (d-1,i)\longrightarrow (0,i+d+z).
\]
All second coordinates are read modulo $q$.  Let $S_t\subseteq\mathbb Z_q$ be the set of transfer positions in block $t$.  The main support theorem says that the admissible families $S_0,\ldots,S_{d-1}$ are exactly the following.  Choose
\[
        a_0,\ldots,a_{d-1}\ge 0,
        \qquad a_0+\cdots+a_{d-1}=z,
\]
and choose $\rho\in\mathbb Z_q$.  Then
\[
        \varnothing\ne S_0\subseteq [\rho,\rho+a_0]_q,
\]
and, for $t=1,\ldots,d-1$,
\[
        \varnothing\ne S_t\subseteq
        \left[
        \rho-\sum_{j=1}^t(a_j+1),\,
        \rho-\sum_{j=1}^t(a_j+1)+a_t
        \right]_q .
\]
There is no hidden extra inequality in this formulation.  An older carry notation uses $m_t=a_t+1$, so that $\sum_t m_t=d+z$; the local separation inequalities $m_{t-1}+m_t\le q+1$ then follow automatically from $\sum_t a_t=z\le q-1$.

Once the support is known, the weights have no further combinatorial restrictions.  Write $b_{t,i}$ for the horizontal weight at $(t,i)$.  If $i\notin S_t$, then $b_{t,i}=1$ and no transfer edge is present.  If $i\in S_t$, then $0<b_{t,i}<1$ and the transfer edge has weight $1-b_{t,i}$.  The condition on each horizontal $q$-cycle is simply
\[
        \prod_{i\in S_t} b_{t,i}=\beta=1-\alpha,
        \qquad t=0,
        \ldots,d-1 .
\]
Every such choice gives a stochastic matrix with characteristic polynomial \eqref{eq:typeIIpoly}, and every stochastic matrix with that characteristic polynomial is permutation-similar to one obtained in this way.

The proof is deliberately elementary after the standard Kirkland--\v Smigoc reduction.  We first translate their two-shift normal form $D C_n^{pd}+(I-D)C_n^{pd+1}$ into the block model above.  The key observation is that a transfer cycle with $k$ transfer edges in each block has length congruence
\[
        q\mid (k-1)z .
\]
Since $\gcd(q,z)=1$ and a simple cycle cannot use more than $q$ transfer edges inside one block, every admissible transfer cycle must use exactly one transfer edge from each block. The length condition therefore becomes a constant-carry identity. A circular separation lemma converts that identity into the interval parametrisation above. Coates' formula then gives
\[
        \chi_A(x)=\bigl(x^q-\beta\bigr)^d-Wx^z,
\]
and stochasticity gives $W=\alpha^d$ by evaluating at $x=1$.

The result is a characteristic-polynomial parametrisation for full-degree Type~II realisers. It does not address Type~III realisers or the separate root-continuation and extremality questions for Karpelevich arcs; these boundaries of the result are collected in Section~\ref{sec:consequences}.

The paper is organised as follows. Section \ref{sec:background} recalls Coates' formula and derives the Type II two-choice normal form used below. Section \ref{sec:support} proves the cyclic support theorem. Section \ref{sec:weighted} adds the weights and proves the full Type II parametrisation. Section \ref{sec:consequences} records the sparsest subcase, the dimension of the weight parameter space, and the precise boundary of the result.

\section{Background and Type II normal form}
\label{sec:background}

Throughout, a stochastic matrix is row-stochastic: its entries are nonnegative and each row sum is one.  A weighted directed graph has positive weights on its directed edges.  The weighted digraph of a nonnegative matrix has an edge $i\to j$ precisely when the $(i,j)$ entry is positive, with that entry as its edge weight.

For a positive integer $q$, write $\mathbb Z_q=\mathbb Z/q\mathbb Z$.  If $u\in\mathbb Z$, let
\[
        \langle u\rangle_q\in\{0,1,\ldots,q-1\}
\]
be its least nonnegative residue modulo $q$.  For $a\in\mathbb Z_q$ and an integer $\ell\in\{1,\ldots,q\}$, define the cyclic interval
\[
        [a,a+\ell-1]_q:=\{a,a+1,\ldots,a+\ell-1\}\subseteq\mathbb Z_q .
\]
When integer lifts are being used, the same notation without the subscript $q$ denotes the ordinary integer interval.

\subsection{Coates' formula}

Let $\Gamma$ be a weighted digraph on $n$ vertices with adjacency matrix $A$. A linear subdigraph is a vertex-disjoint union of directed cycles. If $L$ is a linear subdigraph, let $|V(L)|$ be its number of vertices, $\nu(L)$ its number of cycles, and $w(L)$ the product of its edge weights. Coates' formula gives the coefficients of
\[
        \chi_A(x)=\det(xI_n-A)=x^n+k_1x^{n-1}+\cdots+k_n
\]
by
\begin{equation}
        k_j=\sum_{L\in\mathcal L_j}(-1)^{\nu(L)}w(L),
        \qquad j=1,\ldots,n,
        \label{eq:coates}
\end{equation}
where $\mathcal L_j$ is the set of all linear subdigraphs using exactly $j$ vertices \cite{Coates1959}.

\subsection{External inputs, Farey data, and the Type II block form}

Fix
\[
        q\ge 2,
        \qquad d\ge 2,
        \qquad 1\le z\le q-1,
        \qquad \gcd(q,z)=1,
\]
and put
\[
        n=qd,
        \qquad s=qd-z,
        \qquad c=d+z .
\]
Let
\[
        f_\alpha(x)=\bigl(x^q-\beta\bigr)^d-\alpha^d x^z,
        \qquad \beta=1-\alpha,
        \qquad 0<\alpha<1 .
\]
We first spell out the elementary Farey arithmetic behind these parameters.

\begin{lemma}[Type II Farey data from $q,d,z$]
\label{lem:typeII-farey}
Let $q\ge2$, $d\ge2$, $1\le z\le q-1$, and $\gcd(q,z)=1$.  Put $n=qd$ and $s=qd-z$.  Let $p\in\{1,\ldots,q-1\}$ be the unique residue satisfying
\begin{equation}
        pz\equiv1\pmod q,
        \label{eq:pz}
\end{equation}
and define
\[
        r=\frac{ps+1}{q}.
\]
Then $r$ is an integer, $0<p<q$, $0<r<s$, and
\begin{equation}
        rq-ps=1 .
        \label{eq:farey-determinant}
\end{equation}
Consequently $p/q$ and $r/s$ are Farey neighbours of both orders $n-1$ and $n$.  The corresponding order-$n$ reduced Ito polynomial is the full-degree Type II polynomial
\[
        \bigl(x^q-\beta\bigr)^d-\alpha^d x^z,
        \qquad n=qd,
        \qquad s=qd-z .
\]
In particular, the congruence \eqref{eq:pz} is not an extra convention: it is exactly the determinant-one Farey condition written in Type II variables.
\end{lemma}

\begin{proof}
Since $\gcd(q,z)=1$, the inverse $p$ exists and is unique in $\{1,\ldots,q-1\}$.  Because $s=qd-z$, we have
\[
        ps+1=p(qd-z)+1=pqd-(pz-1),
\]
which is divisible by $q$ by \eqref{eq:pz}.  Thus $r$ is an integer and \eqref{eq:farey-determinant} holds.  The positivity of $p$ gives $r>0$, and
\[
        qs-(ps+1)=(q-p)s-1\ge s-1>0,
\]
because $p\le q-1$ and $s>1$.  Hence $0<r<s$.  Equation \eqref{eq:farey-determinant} implies both $p/q<r/s$ and $\gcd(r,s)=1$.  Finally,
\[
        q<s\le n-1,
        \qquad q+s=q+qd-z=n+(q-z)>n,
\]
so $p/q$ and $r/s$ are Farey neighbours in both $F_{n-1}$ and $F_n$.  Ito's order-$n$ reduction for this full-degree Type II case gives $z=qd-s$ and hence the displayed polynomial.
\end{proof}

For the rest of the paper, $p,r,s$ are the Farey data supplied by Lemma \ref{lem:typeII-farey}.  The same endpoints occur already in order $n-1$, but the boundary arcs for orders $n-1$ and $n$ are not the same.  The next lemma proves the exact-order statement needed for the Dmitriev--Dynkin/Kirkland--\v Smigoc reduction.

\begin{lemma}[Exact order of full-degree Type II arcs]
\label{lem:typeII-exact-order}
Let $q\ge2$, $d\ge2$, $1\le z\le q-1$, and $\gcd(q,z)=1$.  Put
\[
        n=qd,
        \qquad s=qd-z,
\]
and let $p,r$ be the integers from Lemma \ref{lem:typeII-farey}.  Let $K_N(\{q,s\})$ denote the open Karpelevich boundary arc of $\Theta_N$ attached to the Farey neighbours $p/q<r/s$, whenever this pair is a Farey pair of order $N$.  Then every non-endpoint point of the order-$n$ arc is outside $\Theta_{n-1}$:
\[
        K_n(\{q,s\})\subseteq \mathbb C\setminus\Theta_{n-1} .
\]
Consequently, for every $0<\alpha<1$, the Karpelevich--Ito boundary root $\lambda_\alpha$ of
\[
        (x^q-(1-\alpha))^d-\alpha^d x^z
\]
on this open arc satisfies
\[
        \lambda_\alpha\in K_n(\{q,s\})\setminus\Theta_{n-1}.
\]
\end{lemma}

\begin{proof}
Lemma \ref{lem:typeII-farey} gives $q,s\le n-1$ and $q+s>n$, so the same determinant-one pair $p/q<r/s$ is a Farey-neighbour pair for both $F_{n-1}$ and $F_n$.  The distinction is the Ito floor parameter:
\[
        \left\lfloor\frac{n-1}{q}\right\rfloor=d-1,
        \qquad
        \left\lfloor\frac n q\right\rfloor=d .
\]
We compare the two radial boundary radii in the common sector
\[
        \frac{2\pi p}{q}<\theta<\frac{2\pi r}{s} .
\]
Set
\[
        \phi=q\theta-2\pi p,
        \qquad
        \eta=2\pi r-s\theta .
\]
Then
\[
        0<\phi<\frac{2\pi}{s},
        \qquad
        0<\eta<\frac{2\pi}{q},
        \qquad
        s\phi+q\eta=2\pi .
\]
By the radial boundary theorem of Kirkland--Laffey--\v Smigoc \cite[Theorem 1.2]{KirklandLaffeySmigoc2020}, applied to the Farey-neighbour pair $p/q<r/s$ at the relevant order, the boundary on this determinant-one sector is radial. If the Ito floor parameter at that order is $m$, its radius is the unique boundary zero $\rho_m=\rho_m(\theta)$ in $(0,1)$ of
\begin{equation}
        G_m(\rho)
        =
        \rho^q\sin\frac{\eta}{m}
        +
        \rho^{s/m}\sin\phi
        -
        \sin\left(\phi+\frac{\eta}{m}\right).
        \label{eq:radial-Gm}
\end{equation}
The strict upper bound follows also from the Farey property: an interior point of the sector on the unit circle would be a root of unity of order at most the matrix order, contrary to $p/q$ and $r/s$ being consecutive in the corresponding Farey sequence. Thus $m=d-1$ gives the boundary radius of $\Theta_{n-1}$ on this ray, even though that lower-order boundary equation is not full-degree Type II, while $m=d$ gives the boundary radius of $\Theta_n$ on the order-$n$ full-degree Type II arc. We prove
\[
        \rho_d(\theta)>\rho_{d-1}(\theta)
\]
for every interior $\theta$.

Put $m=d-1$, let $\rho=\rho_m(\theta)$, and define
\[
        \delta=\frac{\eta}{m},
        \qquad
        \tau=\frac{m}{m+1} .
\]
Then $0<\delta<\pi$, $0<\tau<1$, and $\eta/(m+1)=\tau\delta$.  Also put
\[
        X=\rho^q,
        \qquad
        Y=\rho^{s/m} .
\]
Since $0<\rho<1$, we have $0<Y<1$ and
\[
        \rho^{s/(m+1)}=Y^\tau .
\]
The equation $G_m(\rho)=0$ is
\begin{equation}
        X\sin\delta+Y\sin\phi=\sin(\phi+\delta).
        \label{eq:old-radius-equation}
\end{equation}
We evaluate the next boundary function at the old radius.  Showing $G_{m+1}(\rho)<0$ is equivalent to showing
\[
        \sin(\phi+\tau\delta)
        -X\sin(\tau\delta)
        -Y^\tau\sin\phi
        >0 .
\]
Using \eqref{eq:old-radius-equation} to eliminate $X$, the left-hand side becomes
\[
        \sin\phi
        \left[
        \frac{\sin((1-\tau)\delta)}{\sin\delta}
        +
        Y\frac{\sin(\tau\delta)}{\sin\delta}
        -
        Y^\tau
        \right].
\]
The function $u\mapsto \sin u/u$ is strictly decreasing on $(0,\pi)$, so
\[
        \frac{\sin((1-\tau)\delta)}{\sin\delta}>1-\tau,
        \qquad
        \frac{\sin(\tau\delta)}{\sin\delta}>\tau .
\]
Weighted AM--GM gives
\[
        Y^\tau<(1-\tau)+\tau Y,
\]
because $0<Y<1$.  Hence the bracket is strictly positive, and therefore
\[
        G_{m+1}(\rho)<0 .
\]
Moreover
\[
        G'_{m+1}(\rho)
        =
        q\rho^{q-1}\sin\frac{\eta}{m+1}
        +
        \frac{s}{m+1}\rho^{s/(m+1)-1}\sin\phi
        >0
\]
for $\rho>0$. At the other end of the radial interval,
\[
\begin{aligned}
G_{m+1}(1)
&=\sin\frac{\eta}{m+1}+\sin\phi
  -\sin\left(\phi+\frac{\eta}{m+1}\right)\\
&=4\sin\frac{\phi}{2}
    \sin\frac{\eta}{2(m+1)}
    \sin\left(\frac{\phi}{2}+\frac{\eta}{2(m+1)}\right)>0.
\end{aligned}
\]
All three sine factors are positive because
$0<\phi<2\pi/s$, $0<\eta/(m+1)<2\pi/(q(m+1))$, and their sum is less than $2\pi$. Since $G_{m+1}(\rho)<0$, continuity and strict monotonicity place its unique zero in $(\rho,1)$. Hence
\[
        \rho_{m+1}(\theta)>\rho_m(\theta).
\]
With $m=d-1$, this is exactly
\[
        \rho_d(\theta)>\rho_{d-1}(\theta).
\]

The radial boundary theorem identifies $\rho_{d-1}(\theta)e^{i\theta}$ as the outer boundary point of $\Theta_{n-1}$ on this ray.  Therefore the order-$n$ point $\rho_d(\theta)e^{i\theta}$ is not in $\Theta_{n-1}$.  Since this holds for every interior argument of the arc, the open order-$n$ Type II arc lies in $\Theta_n\setminus\Theta_{n-1}$.  Ito's parametrisation then gives, for every $0<\alpha<1$, a nonzero boundary root $\lambda_\alpha$ of the displayed Type II polynomial on this open exact-order arc.
\end{proof}

\begin{figure}[t]
\centering
\includegraphics[width=0.74\textwidth]{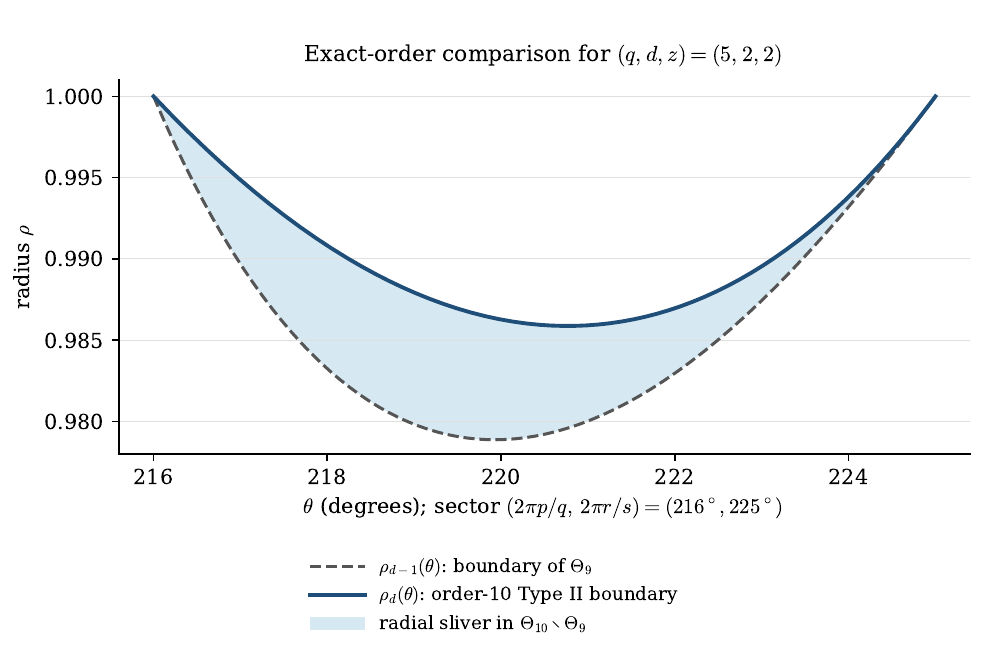}
\caption{The radial comparison in Lemma~\ref{lem:typeII-exact-order} for
$(q,d,z)=(5,2,2)$, so $n=10$, $s=8$, $p/q=3/5$, and $r/s=5/8$.
In the common sector, the order-$10$ full-degree Type~II boundary
(solid) lies strictly outside the boundary of $\Theta_9$ (dashed).
The shaded sliver illustrates the strict inclusion proved in the lemma.}
\label{fig:arcs}
\end{figure}

Ito's parametrisation gives the order-$n$ full-degree Type II arc by the nonzero solutions of
\begin{equation}
        t^s(t^q-\beta)^d=\alpha^d t^{qd},
        \qquad 0<\alpha<1,
        \label{eq:ito-typeII-arc}
\end{equation}
where $d=n/q$.  Since $s=qd-z$, division by the nonzero factor $t^s=t^{qd-z}$ gives
\[
        (t^q-\beta)^d-\alpha^d t^z=0 .
\]
By Lemma \ref{lem:typeII-exact-order}, the Karpelevich--Ito boundary root $\lambda_\alpha$ of $f_\alpha$ satisfies
\[
        \lambda_\alpha\in K_n(\{q,s\})\setminus\Theta_{n-1}.
\]
Moreover, the Farey determinant gives
\[
        \frac rs-\frac pq=\frac1{qs}.
\]
Since $s=qd-z>d$, we have $1/(qs)<1/(qd)=1/n$. As
$p/q=pd/n$, every interior arc argument therefore satisfies
\[
        \frac{2\pi pd}{n}
        <\arg\lambda_\alpha
        <\frac{2\pi(pd+1)}{n},
\]
so it lies strictly between two consecutive $n$th-root arguments.
Since any stochastic matrix with characteristic polynomial $f_\alpha$ has $\lambda_\alpha$ as an eigenvalue, the hypotheses of the Dmitriev--Dynkin/Kirkland--\v Smigoc two-shift reduction apply.

The actual structural input for the next proposition is Kirkland--\v Smigoc, not Karpelevich--Ito.  We import the following results from their paper, and no non-sparse Type II support classification.
\begin{enumerate}[label=\textup{(I\arabic*)},leftmargin=*]
\item Kirkland--\v Smigoc's Theorem 3.1, following Dmitriev--Dynkin
\cite{DmitrievDynkin1946}, says that an $n\times n$ stochastic matrix with an
eigenvalue $\lambda\in\Theta_n\setminus\Theta_{n-1}$ whose argument lies
between two consecutive $n$th-root arguments is permutation-similar to a
two-consecutive-shift matrix.  Their Proposition 3.1 says that, for a
degree-$n$ reduced Ito polynomial with parameters $(q,s,d)$ and such a
boundary root, the only possible cycle lengths are $s$ and $kq$,
$1\le k\le d$, and that at least one $s$-cycle and one $q$-cycle occur
\cite[Theorem 3.1 and Proposition 3.1]{KirklandSmigoc2022}.

\item In the Type II case, their Lemma 6.1 applies to an eigenvalue
$\lambda\in K_n(\{q,s\})\setminus\Theta_{n-1}$ and identifies the two
consecutive shifts as $pd$ and $pd+1$.  Hence every Type II realiser is
permutation-similar to a matrix whose positive support is contained in
\begin{equation}
        D C_n^{pd}+(I-D)C_n^{pd+1},
        \label{eq:KSnormal}
\end{equation}
for a diagonal matrix $D$ with entries in $[0,1]$.  Their Lemma 6.2
identifies the $pd$-shift edges as $d$ disjoint $q$-cycles and says that
no $pd+1$ edge lies on a $q$-cycle.  Their Corollary 3.1 and Lemma 6.3
then force all $d$ of those $q$-cycles to occur with common weight
$\beta$, and rule out cycles of length $kq$ for $k\ge2$
\cite[Corollary 3.1 and Lemmas 6.1--6.3]{KirklandSmigoc2022}.  Combining
this with their Proposition 3.1 leaves only cycles of length
$s=qd-z$ besides the $d$ forced $q$-cycles.  Proposition
\ref{prop:blocknormal} below only relabels this imported Type II normal
form into block coordinates; the new finite support argument begins after
that relabelling.

\item From Joshi--Kirkland--\v Smigoc we import only contextual information \cite[Theorem 3.10, Corollary 3.12, and Section 5]{JoshiKirklandSmigoc2023}: their Theorem 3.10 and Corollary 3.12 classify when Karpelevich arcs are powers of other arcs, and their Section 5 treats powers of sparsest realising matrices.  Their results are not used in the proof of Theorem \ref{thm:weighted}; they explain why the non-sparse Type II class classified here is outside the already-treated powers-of-sparsest framework.

\item The companion Type III manuscript is cited only for context; no result from it is used in the present proof \cite{VerbekenGinisTypeIII}.
\end{enumerate}

The next proposition converts the imported Type~II two-shift form \eqref{eq:KSnormal} into the block coordinates used below.

\begin{proposition}[Type II block normal form]
\label{prop:blocknormal}
After permutation similarity, every stochastic realiser of $f_\alpha$ is supported on a weighted two-choice digraph with vertex set
\[
        V=\mathbb Z_d\times\mathbb Z_q,
\]
positive forced horizontal edges
\begin{equation}
        h_{t,i}:(t,i)\longrightarrow(t,i+1),
        \label{eq:horiz}
\end{equation}
possible transfer edges
\begin{equation}
        g_{t,i}:(t,i)\longrightarrow(t+1,i),
        \qquad t=0,\ldots,d-2,
        \label{eq:transfer1}
\end{equation}
and possible final transfer edges
\begin{equation}
        g_{d-1,i}:(d-1,i)\longrightarrow(0,i+c).
        \label{eq:transferlast}
\end{equation}
All second coordinates are read modulo $q$.  Moreover, every horizontal
edge is positive, and in each block the product of the horizontal edge
weights is $\beta$.
\end{proposition}

\begin{proof}
Use zero-based labels in $\mathbb Z_n$.  Define
\[
        \phi:\mathbb Z_d\times\mathbb Z_q\longrightarrow\mathbb Z_n,
        \qquad
        \phi(t,i)=t(pd+1)+ipd \pmod n .
\]
This map is bijective.  Indeed, if $t(pd+1)+ipd\equiv 0\pmod {qd}$, then reduction modulo $d$ gives $t\equiv0\pmod d$, and then $ipd\equiv0\pmod {qd}$ gives $i\equiv0\pmod q$ because $\gcd(p,q)=1$.  Applying this zero-kernel argument to the difference of two pairs shows that $\phi$ is injective, hence bijective.

Under this relabelling, the shift $pd$ becomes
\[
        \phi(t,i)+pd=\phi(t,i+1),
\]
so it gives the horizontal edges \eqref{eq:horiz}.  For $t=0,\ldots,d-2$, the shift $pd+1$ becomes
\[
        \phi(t,i)+pd+1=\phi(t+1,i),
\]
so it gives \eqref{eq:transfer1}.  For the last block,
\[
        \phi(d-1,i)+pd+1=d(pd+1)+ipd .
\]
Using \eqref{eq:pz},
\[
        d(pd+1)\equiv (d+z)pd\pmod {qd},
\]
and hence
\[
        \phi(d-1,i)+pd+1=\phi(0,i+d+z),
\]
which gives \eqref{eq:transferlast}.  Finally, since $0<\beta<1$, the Type II consequences imported above from
Kirkland--\v Smigoc force each of the $d$ horizontal $q$-cycles to occur
with positive product $\beta$.  Thus no horizontal edge on these cycles is
zero.  Lemma 6.2 identifies the horizontal cycles, while Corollary 3.1 and
Lemma 6.3 force their common weight to be $\beta$ and rule out the other
$kq$-cycles.
\end{proof}

Let
\[
        S_t\subseteq\mathbb Z_q
\]
be the set of indices $i$ for which the transfer edge $g_{t,i}$ is present.  The sets $S_t$ are nonempty in every non-degenerate Type II realiser, because the horizontal $q$-cycle in each block has weight $\beta<1$.

By the reduction input just stated, after the $d$ forced horizontal $q$-cycles have been accounted for, every other directed cycle in a full-degree Type II realiser has length
\[
        s=qd-z .
\]
The remaining non-sparse Type II question is therefore the following finite combinatorial support problem.

\begin{definition}[Type II admissible support]
\label{def:admissible}
A family of nonempty sets
\[
        S_0,\ldots,S_{d-1}\subseteq\mathbb Z_q
\]
is called Type II admissible if the two-choice digraph \eqref{eq:horiz}--\eqref{eq:transferlast} has exactly the $d$ horizontal $q$-cycles and every other directed cycle has length $s=qd-z$.
\end{definition}

Sections \ref{sec:support} and \ref{sec:weighted} solve Definition \ref{def:admissible} and then put the weights back in.

\section{The Type II support theorem}
\label{sec:support}

\subsection{Cycle lengths in the two-choice graph}

A directed cycle that uses transfer edges must use the same number of transfer edges in each block, because every transfer edge advances the first coordinate by one modulo $d$.  The following elementary observation justifies using arbitrary transfer tuples in the necessity proof.

\begin{lemma}[Transfer tuples are simple cycles]
\label{lem:tuple-simple}
Let $S_0,\ldots,S_{d-1}$ be nonempty transfer sets in the two-choice graph \eqref{eq:horiz}--\eqref{eq:transferlast}.  For every tuple
\[
        (i_0,\ldots,i_{d-1})\in S_0\times\cdots\times S_{d-1},
\]
there is a simple directed cycle using exactly the transfer edges
\[
        g_{0,i_0},g_{1,i_1},\ldots,g_{d-1,i_{d-1}}.
\]
Its length is
\begin{equation}
        L(i_0,\ldots,i_{d-1})
        =d+\sum_{t=1}^{d-1}\langle i_t-i_{t-1}\rangle_q
          +\langle i_0-i_{d-1}-c\rangle_q .
        \label{eq:lengthone}
\end{equation}
\end{lemma}

\begin{proof}
Start at $(0,i_0)$, take the transfer edge $g_{0,i_0}$ to $(1,i_0)$, then follow horizontal edges in block $1$ until $(1,i_1)$ is reached.  Continue in this way: after $g_{t-1,i_{t-1}}$ the walk is at $(t,i_{t-1})$, follows the horizontal $q$-cycle in block $t$ for $\langle i_t-i_{t-1}\rangle_q$ steps until $(t,i_t)$, and then takes $g_{t,i_t}$, for $t=1,\ldots,d-1$.  The last transfer edge lands at $(0,i_{d-1}+c)$, and the walk closes by following the horizontal edges in block $0$ for $\langle i_0-i_{d-1}-c\rangle_q$ steps back to $(0,i_0)$.

Each horizontal segment has length at most $q-1$, so it has no repeated vertex inside its own block.  No block other than block $0$ is visited more than once, because the first coordinate strictly advances through $1,2,\ldots,d-1$ before returning to $0$.  In block $0$, the starting vertex $(0,i_0)$ reappears only as the terminal vertex of the closing segment; the closing segment also has length at most $q-1$, so it cannot wrap once around the whole horizontal $q$-cycle before closing.  Hence the closed walk is a simple directed cycle.  Counting the $d$ transfer edges and the displayed horizontal segment lengths gives \eqref{eq:lengthone}.
\end{proof}

Such a tuple-cycle has Type II length $s=qd-z$ precisely when
\begin{equation}
        \sum_{t=1}^{d-1}\langle i_t-i_{t-1}\rangle_q
          +\langle i_0-i_{d-1}-c\rangle_q
        =dq-c .
        \label{eq:constantcarry}
\end{equation}
We call \eqref{eq:constantcarry} the constant-carry identity.

\begin{lemma}[Non-horizontal cycles use one transfer edge per block]
\label{lem:kone}
If a Type II admissible support has a directed cycle using at least one transfer edge, and if it uses $k$ transfer edges from each block, then $k=1$.
\end{lemma}

\begin{proof}
Let $h$ be the number of horizontal edges in such a cycle.  The length of the cycle is $kd+h$.  Because the cycle uses a transfer edge, it is not one of the horizontal $q$-cycles.  Since the support is Type II admissible, its length is therefore $s=qd-z$.  On the other hand, each full passage through the $d$ blocks shifts the second coordinate by $c=d+z$, while each horizontal edge shifts it by one.  Closure modulo $q$ gives
\[
        h+kc\equiv0\pmod q .
\]
Using $kd+h=qd-z$, we obtain
\[
        h+kc=qd-z-kd+k(d+z)=qd+(k-1)z .
\]
Therefore $q\mid (k-1)z$.  Since $\gcd(q,z)=1$, $k\equiv1\pmod q$.  A simple cycle cannot use more than $q$ transfer edges from one block, because a block has only $q$ vertices.  Hence $k=1$.
\end{proof}

Combining Lemma \ref{lem:tuple-simple} with admissibility, every tuple $(i_0,\ldots,i_{d-1})\in S_0\times\cdots\times S_{d-1}$ in a Type II admissible support satisfies \eqref{eq:constantcarry}.

\subsection{A circular separation lemma}

For $a,b\in\mathbb Z_q$, write
\[
        D(a,b)=\langle b-a\rangle_q .
\]
The directed arc from $a$ to $b$ is
\[
        [a,b]_q:=\{a,a+1,\ldots,a+D(a,b)\}\subseteq\mathbb Z_q .
\]
A point $r\in\mathbb Z_q$ is a common blocker for a family of arcs if $r$ belongs to every arc in the family.
The four-point identity below is a cyclic equality form of the classical Monge property; see \cite{BurkardKlinzRudolf1996} for background on Monge structures in combinatorial optimisation.

\begin{lemma}[Circular Monge separation]
\label{lem:separation}
Let $A,B\subseteq\mathbb Z_q$ be nonempty.  The following are equivalent.
\begin{enumerate}[label=\textup{(\roman*)}]
\item For all $a,a'\in A$ and $b,b'\in B$,
\begin{equation}
        D(a,b)+D(a',b')=D(a,b')+D(a',b).
        \label{eq:monge}
\end{equation}
\item The family of directed arcs $[a,b]_q$, $a\in A$, $b\in B$, has a common blocker.
\item There are integer representatives $\widehat a\equiv a\pmod q$, $\widehat b\equiv b\pmod q$, chosen for all $a\in A$ and $b\in B$, such that
\begin{equation}
        \widehat b<\widehat a\le \widehat b+q
        \qquad(a\in A,\ b\in B),
        \label{eq:separated-lifts-order}
\end{equation}
 and hence
\begin{equation}
        D(a,b)=q+\widehat b-\widehat a
        \qquad(a\in A,\ b\in B).
        \label{eq:separated-lifts-distance}
\end{equation}
\item There are integer intervals $I_B=[u,v]$ and $I_A=[v+1,w]$, with $w-u\le q$, whose residue classes contain $B$ and $A$, respectively.  Equivalently, $B$ and $A$ are contained in consecutive cyclic intervals whose total cardinality is at most $q+1$, with the interval containing $B$ immediately preceding the interval containing $A$.
\end{enumerate}
\end{lemma}

\begin{proof}
Assume first that $\rho$ is a common blocker.  Rotate the circle so that $\rho=0$.  For $b\in B$, take $\widehat b\in\{0,1,\ldots,q-1\}$.  For $a\in A$, take $\widehat a\in\{1,\ldots,q\}$, using the representative $q$ when $a=0$.  Since $0\in[a,b]_q$ for every $a,b$, these representatives satisfy $\widehat b<\widehat a\le\widehat b+q$, and the directed distance from $a$ to $b$ is exactly $q+\widehat b-\widehat a$.  Thus (ii) implies (iii).  Under (iii), both sides of \eqref{eq:monge} equal
\[
        2q+\widehat b+\widehat b'-\widehat a-\widehat a',
\]
so (iii) implies (i).  Condition (iii) also gives a common blocker: if $M=\max\{\widehat a:a\in A\}$, then for every $a\in A$ and $b\in B$ the lifted directed path from $\widehat a$ to $\widehat b+q$ contains $M$.  Hence the residue class of $M$ lies in every $[a,b]_q$.

It remains only to prove that (i) forces a common blocker.  We prove the contrapositive.  Suppose the arcs $[a,b]_q$ have no common blocker.  Choose one of them, say $I=[a_0,b_0]_q$, with $D(a_0,b_0)$ minimal.  Rotate and lift so that
\[
        a_0=0,
        \qquad b_0=m,
        \qquad 0\le m\le q-1,
\]
so $I=\{0,1,\ldots,m\}$.  If some arc in the family is disjoint from $I$, we already have two disjoint arcs.  Otherwise every arc meets $I$.  Since the intersection of the whole family is empty and is contained in $I$, there is an arc $J_R=[u,v]_q$ missing $0$ and an arc $J_L=[w,y]_q$ missing $m$.  Minimality of $I$ and the assumption that all arcs meet $I$ force, in the chosen lift,
\[
        0<u\le m\le v<q,
        \qquad
        m<w<q,
        \qquad
        0\le y<m .
\]
The cross arc $[u,y]_q$ also belongs to the family.  If $u\le y$, then $D(u,y)<m$, contradicting minimality; hence $y<u$.  If $w\le v$, then the other cross arc $[w,v]_q$ is contained in the complement of $I$, so it is disjoint from $I$.  If $v<w$, then $J_R$ and $J_L$ are disjoint, because $y<u\le m\le v<w$.  Thus in all cases there are arcs $[a,b]_q$ and $[a',b']_q$ in the family that are disjoint.

Cut the circle before the first of these disjoint arcs and choose lifts with
\[
        a\le b<a'\le b'<a+q .
\]
Then
\[
\begin{aligned}
        D(a,b)+D(a',b')&=(b-a)+(b'-a'),\\
        D(a,b')+D(a',b)&=(b'-a)+(b+q-a')
\end{aligned}
\]
and the second expression is larger than the first by $q$.  Hence \eqref{eq:monge} fails.  This proves (i)$\Rightarrow$(ii).

Finally, (iii) and (iv) are the same separation written with or without explicit representatives.  From (iii), let $u=\min\widehat B$, $v=\max\widehat B$, and $w=\max\widehat A$; then every representative of $A$ is at least $v+1$, and $w-u\le q$.  Conversely, representatives in intervals $[u,v]$ and $[v+1,w]$ with $w-u\le q$ satisfy \eqref{eq:separated-lifts-order} and \eqref{eq:separated-lifts-distance}.
\end{proof}

\subsection{The support classification}

We state the pure support theorem in composition form. The carry notation used in the proof is recorded separately after the theorem.

\begin{theorem}[Type II support theorem]
\label{thm:support}
Let $q\ge2$, $d\ge2$, $1\le z\le q-1$, and $\gcd(q,z)=1$. Put $c=d+z$ and $s=qd-z$. Let $S_0,\ldots,S_{d-1}$ be nonempty subsets of $\mathbb Z_q$ in the Type II two-choice digraph \eqref{eq:horiz}--\eqref{eq:transferlast}. Then the support is Type II admissible if and only if there exist
\[
        a_0,\ldots,a_{d-1}\in\mathbb Z_{\ge0},
        \qquad a_0+\cdots+a_{d-1}=z,
\]
and $\rho\in\mathbb Z_q$ such that
\begin{equation}
        S_0\subseteq [\rho,\rho+a_0]_q,
        \label{eq:S0a}
\end{equation}
\begin{equation}
        S_t\subseteq
        \left[
        \rho-\sum_{j=1}^t(a_j+1),\,
        \rho-\sum_{j=1}^t(a_j+1)+a_t
        \right]_q,
        \qquad t=1,\ldots,d-1 .
        \label{eq:Sta}
\end{equation}
\end{theorem}

\begin{proof}
\noindent\emph{Sufficiency.}
Assume that the displayed composition and interval conditions hold, and put
\[
        m_t=a_t+1,
        \qquad t\in\mathbb Z_d.
\]
Then
\[
        \sum_{t=0}^{d-1}m_t=d+z=c,
        \qquad
        m_t\le z+1\le q,
        \qquad
        m_{t-1}+m_t=2+a_{t-1}+a_t\le 2+z\le q+1,
\]
so the bounds used below are automatic consequences of the composition. Choose integer lifts
\[
        R_0=\rho,
        \qquad
        R_t=\rho-\sum_{j=1}^t m_j\quad (t=1,\ldots,d-1),
\]
and
\[
        R_d=R_0-(d+z)=R_0-c .
\]
Then
\begin{equation}
        S_t\subseteq [R_t,R_{t-1}-1]_q,
        \qquad t=1,\ldots,d-1,
        \label{eq:suff-incl1}
\end{equation}
and, after translating the interval for $S_0$ by $-c$,
\begin{equation}
        S_0-c\subseteq [R_d,R_{d-1}-1]_q .
        \label{eq:suff-incl0}
\end{equation}
For each adjacent pair of transfer sets, the target set and the source set are contained in consecutive cyclic intervals of lengths $m_t$ and $m_{t-1}$, with indices read modulo $d$. By the preceding local bound and Lemma \ref{lem:separation}, all horizontal arcs in that block have a common blocker. Explicitly, this applies to the arcs from $S_{t-1}$ to $S_t$ for $t=1,\ldots,d-1$, and to the arcs from $S_{d-1}$ to $S_0-c$, equivalently from $S_{d-1}+c$ to $S_0$, in block $0$.

A simple directed cycle cannot use two transfer edges with target in the
same block.  Indeed, between the two transfer targets and the two subsequent
transfer tails, the cycle would contain two horizontal subpaths in that
block with disjoint vertex sets; Lemma \ref{lem:separation} gives a common
blocker for all such subpaths, a contradiction.  Since the number of
transfer edges used in each block is the same, every transfer cycle uses
exactly one transfer edge from each block.

It remains to compute its length.  Let $i_t\in S_t$.  Choose lifts so that
\[
        \widetilde i_0\in [R_0,R_0+m_0-1],
        \qquad
        \widetilde i_t\in [R_t,R_{t-1}-1]\quad(t=1,\ldots,d-1),
\]
so that $\widetilde i_0-c\in [R_d,R_{d-1}-1]$. The interval inclusions, together with the preceding local bound, give
\[
        1\le \widetilde i_{t-1}-\widetilde i_t
        \le m_{t-1}+m_t-1\le q,
        \qquad t=1,\ldots,d-1,
\]
and
\[
        1\le \widetilde i_{d-1}-(\widetilde i_0-c)
        \le m_{d-1}+m_0-1\le q.
\]
Hence
\[
        \langle i_t-i_{t-1}\rangle_q=q+\widetilde i_t-\widetilde i_{t-1},
        \qquad t=1,\ldots,d-1,
\]
and
\[
        \langle i_0-i_{d-1}-c\rangle_q=q+\widetilde i_0-c-\widetilde i_{d-1}.
\]
Substitution in \eqref{eq:lengthone} gives
\[
\begin{aligned}
        L
        &=d+\sum_{t=1}^{d-1}(q+\widetilde i_t-\widetilde i_{t-1})
          +(q+\widetilde i_0-c-\widetilde i_{d-1})  \\
        &=d+dq-c
        =qd-z .
\end{aligned}
\]
Thus every transfer cycle has length $s=qd-z$, and the only cycles without transfer edges are the $d$ horizontal $q$-cycles.  The support is admissible.

\medskip
\noindent\emph{Necessity.}
Assume the support is Type II admissible.  By Lemma \ref{lem:tuple-simple}, every tuple $(i_0,\ldots,i_{d-1})\in S_0\times\cdots\times S_{d-1}$ gives a simple directed transfer cycle using exactly one transfer edge from each block.  Since the support is admissible, every such cycle has length $s=qd-z$.  Consequently \eqref{eq:constantcarry} holds for every tuple.

It is convenient to put
\[
        T_t=S_t\quad(0\le t\le d-1),
        \qquad
        T_d=S_0-c,
\]
and, for $t=1,\ldots,d$, to write
\[
        E_t(x,y)=D(x,y)=\langle y-x\rangle_q,
        \qquad x\in T_{t-1},\ y\in T_t .
\]
Here an element $y\in T_d$ is always of the form $y=i_0-c$ with $i_0\in S_0$.  The constant-carry identity becomes
\begin{equation}
        \sum_{t=1}^d E_t(x_{t-1},x_t)=dq-c,
        \qquad x_t\in T_t,
        \qquad x_d=x_0-c .
        \label{eq:Tconstant}
\end{equation}

We first prove that every adjacent pair $(T_{t-1},T_t)$ satisfies the Monge identity \eqref{eq:monge}.  Suppose $d\ge3$.  Fix $t\in\{1,\ldots,d\}$, and choose
\[
        a,a'\in T_{t-1},
        \qquad
        b,b'\in T_t .
\]
We show that
\[
        E_t(a,b)+E_t(a',b')=E_t(a,b')+E_t(a',b).
\]
There are three cases, because the cyclic constraint $x_d=x_0-c$ makes the endpoints look slightly different.

If $2\le t\le d-1$, fix all variables except $x_{t-1}$ and $x_t$; in particular fix $u\in T_{t-2}$ and $v\in T_{t+1}$.  In \eqref{eq:Tconstant}, the part depending on $x_{t-1}=a$ and $x_t=b$ is
\[
        E_{t-1}(u,a)+E_t(a,b)+E_{t+1}(b,v)+C,
\]
where $C$ is independent of $a,b$.  Since the whole sum in \eqref{eq:Tconstant} is the same for the four choices $(a,b)$, $(a,b')$, $(a',b)$, $(a',b')$, adding the identities for $(a,b)$ and $(a',b')$ and subtracting the identities for $(a,b')$ and $(a',b)$ cancels the two neighbouring terms and $C$, leaving exactly the Monge identity for $E_t$.

For $t=1$, the variable $a=x_0\in T_0$ also determines the endpoint $x_d=a-c\in T_d$.  Fix choices $x_2^*,\ldots,x_{d-1}^*$, with the evident interpretation that for $d=3$ there is only the single fixed choice $x_2^*$.  The part of \eqref{eq:Tconstant} depending on $a\in T_0$ and $b\in T_1$ is
\[
        E_d(x_{d-1}^*,a-c)+E_1(a,b)+E_2(b,x_2^*)+C .
\]
The same four-corner cancellation cancels $E_d(x_{d-1}^*,a-c)$, $E_2(b,x_2^*)$, and $C$, and gives the Monge identity for $E_1$.

For $t=d$, write $b\in T_d$ and remember that the corresponding initial variable is $x_0=b+c\in T_0$.  Fix choices $x_1^*,\ldots,x_{d-2}^*$, again with only one fixed choice when $d=3$.  The part depending on $a\in T_{d-1}$ and $b\in T_d$ is
\[
        E_{d-1}(x_{d-2}^*,a)+E_d(a,b)+E_1(b+c,x_1^*)+C .
\]
Again the four-corner cancellation gives the Monge identity for $E_d$.  Thus every adjacent pair satisfies \eqref{eq:monge} when $d\ge3$.

Now let $d=2$, and write $A=S_0$ and $B=S_1$.  For every $a\in A$ and $b\in B$, the one-transfer cycle has length $s=2q-z$, hence
\[
        D(a,b)+D(b+c,a)=2q-c .
\]
Equivalently,
\begin{equation}
        D(a,b)>q-c
        \qquad(a\in A,\ b\in B).
        \label{eq:d2-threshold}
\end{equation}
Suppose that the arcs $[a,b]_q$, $a\in A$, $b\in B$, have no common blocker.  By the contrapositive part of Lemma \ref{lem:separation}, there are two disjoint such arcs.  After relabelling the two arcs if necessary, rotate the circle and choose integer representatives so that
\[
        a=0\le b<a'\le b'<q .
\]
The inequalities \eqref{eq:d2-threshold} applied to $(a,b)$ and $(a',b')$ give
\[
        b+c>q,
        \qquad
        b'+c-q>a' .
\]
Also $c\le q+1$ and $b<a'$ give
\[
        b+c-q\le b+1\le a' .
\]
Thus in block $1$ the arcs $[a,b]$ and $[a',b']$ are disjoint, while in block $0$ the crossed arcs
\[
        [b+c-q,a']
        \qquad\text{and}\qquad
        [b'+c-q,q]
\]
are disjoint lifts of the arcs from $b+c$ to $a'$ and from $b'+c$ to $a$.  The corresponding cycle is, in lifted notation,
\[
\begin{aligned}
        (0,a)&\to(1,a)\leadsto(1,b)\to(0,b+c)
        \leadsto(0,a')  \\
        &\to(1,a')\leadsto(1,b')\to(0,b'+c)
        \leadsto(0,a),
\end{aligned}
\]
where $\leadsto$ denotes the appropriate horizontal path inside the indicated
block.  The transfer tails are distinct in their respective blocks, and the
transfer heads are precisely the initial vertices of the displayed horizontal
paths, so the transfer edges introduce no additional vertex identifications.
The two block-$1$ horizontal paths are disjoint by $a=0\le b<a'\le b'$, and the
two block-$0$ horizontal paths are disjoint by the displayed crossed intervals.
Hence this is a simple directed cycle using two transfer edges from each block,
contradicting Lemma \ref{lem:kone}.  Thus the arcs from $A$ to $B$ have a common
blocker, and Lemma \ref{lem:separation} gives the Monge identity for $(A,B)$.
The final pair satisfies the Monge identity as well, because
\[
        D(b,a-c)=2q-c-D(a,b)
        \qquad(a\in A,\ b\in B)
\]
by the one-transfer length condition; replacing $D(a,b)$ by this affine negative preserves the $2\times2$ identity.

By Lemma \ref{lem:separation}, choose separated integer lifts $X_0:T_0\to\mathbb Z$ and $X_1:T_1\to\mathbb Z$ for the first adjacent pair, so that
\begin{equation}
        E_1(x_0,x_1)=q+X_1(x_1)-X_0(x_0)
        \qquad(x_0\in T_0,
        \ x_1\in T_1).
        \label{eq:lift1}
\end{equation}
We now propagate these lifts around the cycle.  Suppose $2\le t\le d$ and $X_{t-1}$ has already been defined so that
\[
        E_{t-1}(x_{t-2},x_{t-1})
        =q+X_{t-1}(x_{t-1})-X_{t-2}(x_{t-2}).
\]
Fix $x_{t-2}\in T_{t-2}$ and $x_t\in T_t$ and vary $x_{t-1}\in T_{t-1}$ in \eqref{eq:Tconstant}, keeping all other coordinates fixed.  The terms depending on $x_{t-1}$ are
\[
        q+X_{t-1}(x_{t-1})-X_{t-2}(x_{t-2})
        +E_t(x_{t-1},x_t),
\]
so $X_{t-1}(x_{t-1})+E_t(x_{t-1},x_t)$ is independent of $x_{t-1}$.  Choosing a base point $x_{t-1}^*\in T_{t-1}$ and setting
\[
        X_t(x_t)=X_{t-1}(x_{t-1}^*)+E_t(x_{t-1}^*,x_t)-q
\]
therefore gives
\begin{equation}
        E_t(x_{t-1},x_t)=q+X_t(x_t)-X_{t-1}(x_{t-1})
        \qquad(x_{t-1}\in T_{t-1},
        \ x_t\in T_t).
        \label{eq:liftt}
\end{equation}
This defines compatible lifts $X_t$ for all $t=0,1,\ldots,d$.  Moreover, each $X_t$ is an integer representative of the corresponding residue class:
\begin{equation}
        X_t(x)\equiv x\pmod q,
        \qquad x\in T_t .
        \label{eq:lift-residue}
\end{equation}
This is true for $X_0$ and $X_1$ by the initial choice of separated lifts.  If it is true for $X_{t-1}$, then the defining formula for $X_t$ gives
\[
        X_t(x_t)\equiv x_{t-1}^*+\langle x_t-x_{t-1}^*\rangle_q\equiv x_t\pmod q,
\]
so \eqref{eq:lift-residue} follows by induction.

Summing \eqref{eq:liftt} over $t=1,\ldots,d$ and comparing with \eqref{eq:Tconstant} gives, for every $x_0\in T_0$,
\[
        dq+X_d(x_0-c)-X_0(x_0)=dq-c .
\]
Hence
\begin{equation}
        X_d(x_0-c)=X_0(x_0)-c
        \qquad(x_0\in S_0).
        \label{eq:closure-lifts}
\end{equation}
Now define
\[
        R_t=\min X_t(T_t),
        \qquad t=0,1,\ldots,d .
\]
By \eqref{eq:closure-lifts}, $R_d=R_0-c$.  Moreover, \eqref{eq:liftt} and $0\le E_t\le q-1$ imply
\[
        1\le X_{t-1}(x_{t-1})-X_t(x_t)\le q
        \qquad(x_{t-1}\in T_{t-1},
        \ x_t\in T_t).
\]
Since the $X_t$ are integer representatives by \eqref{eq:lift-residue}, these inequalities imply
\begin{equation}
        T_t\subseteq [R_t,R_{t-1}-1]_q,
        \qquad t=1,\ldots,d .
        \label{eq:nec-lift-inclusion}
\end{equation}
For $t=d$ this says $S_0-c\subseteq [R_d,R_{d-1}-1]_q$; adding $c$ and using $R_d+c=R_0$ gives the required interval for $S_0$.

Define
\[
        m_t=R_{t-1}-R_t,
        \qquad t=1,\ldots,d-1,
\]
and
\[
        m_0=R_{d-1}-R_d .
\]
The preceding inequalities give $m_t\in\{1,\ldots,q\}$.  Also
\[
        \sum_{t=0}^{d-1}m_t=R_0-R_d=c=d+z .
\]
Consequently, writing $a_t=m_t-1$, we have $a_t\ge0$ and $\sum_t a_t=z$. Moreover,
\[
        m_{t-1}+m_t=2+a_{t-1}+a_t\le 2+z\le q+1 .
\]
The inclusions \eqref{eq:nec-lift-inclusion}, together with the translated inclusion for $S_0$, are exactly \eqref{eq:S0a}--\eqref{eq:Sta} with $\rho=R_0$. This completes the proof.
\end{proof}

\begin{corollary}[Carry formulation]
\label{cor:carry}
Under the hypotheses of Theorem~\ref{thm:support}, an admissible support equivalently has an interval presentation determined by
\[
        m_0,\ldots,m_{d-1}\in\{1,\ldots,q\},
        \qquad
        \sum_{t=0}^{d-1}m_t=d+z,
\]
and $\rho\in\mathbb Z_q$, with
\[
        S_0\subseteq[\rho,\rho+m_0-1]_q
\]
and
\[
        S_t\subseteq
        \left[
        \rho-\sum_{j=1}^t m_j,\,
        \rho-\sum_{j=1}^t m_j+m_t-1
        \right]_q,
        \qquad t=1,\ldots,d-1.
\]
The local inequalities
\[
        m_{t-1}+m_t\le q+1,
        \qquad t\in\mathbb Z_d,
\]
are automatic and are not additional hypotheses.
\end{corollary}

\begin{proof}
Set $a_t=m_t-1$. Then $\sum_t a_t=z$ and every $a_t$ is nonnegative. Conversely, from any composition in Theorem~\ref{thm:support}, setting $m_t=a_t+1$ gives $m_t\le z+1\le q$ and
\[
        m_{t-1}+m_t
        =2+a_{t-1}+a_t
        \le z+2\le q+1.
\]
The interval formulas are the same after this substitution.
\end{proof}

\begin{remark}[The two-block asymmetry]
For $d\ge3$, the four-corner cancellation in the necessity proof forces the Monge identity for every adjacent pair directly from the tuple-cycle length identity. When $d=2$, the cyclic endpoints coincide too closely for that cancellation, and the additional disjoint-cycle argument leading to \eqref{eq:d2-threshold} is essential.
\end{remark}

\begin{remark}
The representation by $a_0,\ldots,a_{d-1}$ and $\rho$ is not meant to be unique.  The theorem is existential: a support family is admissible if and only if it has at least one such interval presentation.  Different cyclic origins, or larger intervals containing the same support sets, may represent the same support family.
\end{remark}

\section{The weighted Type II parametrisation}
\label{sec:weighted}

We now add weights.  Let $S_0,\ldots,S_{d-1}$ be a Type II admissible support.  For $i\in S_t$, choose a number
\[
        b_{t,i}\in(0,1),
\]
and for $i\notin S_t$ put $b_{t,i}=1$.  The weighted two-choice matrix is the $n\times n$ matrix indexed by $\mathbb Z_d\times\mathbb Z_q$ with nonzero entries
\begin{equation}
        A_{(t,i),(t,i+1)}=b_{t,i},
        \label{eq:matrixh}
\end{equation}
\begin{equation}
        A_{(t,i),(t+1,i)}=1-b_{t,i},
        \qquad t=0,\ldots,d-2,
        \qquad i\in S_t,
        \label{eq:matrixt}
\end{equation}
and
\begin{equation}
        A_{(d-1,i),(0,i+c)}=1-b_{d-1,i},
        \qquad i\in S_{d-1}.
        \label{eq:matrixtlast}
\end{equation}
All omitted entries are zero.  This matrix is stochastic.

\begin{theorem}[Full Type II realiser parametrisation]
\label{thm:weighted}
Let
\[
        f_\alpha(x)=\bigl(x^q-(1-\alpha)\bigr)^d-\alpha^d x^z,
        \qquad 0<\alpha<1,
\]
where $q\ge2$, $d\ge2$, $1\le z\le q-1$, and $\gcd(q,z)=1$.  Put $\beta=1-\alpha$ and $n=qd$.

An $n\times n$ stochastic matrix has characteristic polynomial $f_\alpha$ if and only if it is permutation-similar to a matrix of the form \eqref{eq:matrixh}--\eqref{eq:matrixtlast}, with $0<b_{t,i}<1$ for $i\in S_t$ and $b_{t,i}=1$ for $i\notin S_t$, where the data satisfy the following conditions.
\begin{enumerate}[label=\textup{(\alph*)}]
\item There are $a_0,\ldots,a_{d-1}\in\mathbb Z_{\ge0}$ and $\rho\in\mathbb Z_q$ such that
\[
        a_0+\cdots+a_{d-1}=z,
\]
\[
        \varnothing\ne S_0\subseteq [\rho,\rho+a_0]_q,
\]
and, for $t=1,\ldots,d-1$,
\[
        \varnothing\ne S_t\subseteq
        \left[
        \rho-\sum_{j=1}^t(a_j+1),\,
        \rho-\sum_{j=1}^t(a_j+1)+a_t
        \right]_q .
\]
\item The weights satisfy $b_{t,i}\in(0,1)$ for $i\in S_t$, $b_{t,i}=1$ for $i\notin S_t$, and
\begin{equation}
        \prod_{i\in S_t}b_{t,i}=\beta,
        \qquad t=0,
        \ldots,d-1 .
        \label{eq:weightproduct}
\end{equation}
\end{enumerate}
The interval data in \textup{(a)} need not be unique, but every realiser admits at least one such presentation.
\end{theorem}

\begin{proof}
First suppose that $A$ is constructed from data satisfying (a) and (b).  By Theorem \ref{thm:support}, its support has only the $d$ horizontal $q$-cycles and transfer cycles of length $s=qd-z$.  The horizontal $q$-cycle in block $t$ has weight
\[
        \prod_{i\in\mathbb Z_q}b_{t,i}
        =\prod_{i\in S_t}b_{t,i}
        =\beta .
\]
The $d$ horizontal cycles are vertex-disjoint.  Therefore their contribution to the characteristic polynomial is exactly
\[
        \bigl(x^q-\beta\bigr)^d .
\]
Every transfer cycle has length $s=qd-z$, and two transfer cycles cannot be vertex-disjoint because
\[
        2s=2(qd-z)>qd=n
\]
under $d\ge2$ and $z\le q-1$.  A transfer cycle also intersects every horizontal $q$-cycle, since it uses at least one vertex in every block.  Hence Coates' formula gives
\[
        \chi_A(x)=\bigl(x^q-\beta\bigr)^d-Wx^z
\]
Here the sign is negative because each transfer cycle is a single cycle in
Coates' formula and has length $s=qd-z$, hence contributes to the
coefficient of $x^{n-s}=x^z$.  The scalar $W\ge0$ is the total weight of
the transfer cycles.  Since $A$ is stochastic, $1$ is an eigenvalue, so
\[
        0=\chi_A(1)=(1-\beta)^d-W=\alpha^d-W .
\]
Thus $W=\alpha^d$, and
\[
        \chi_A(x)=\bigl(x^q-\beta\bigr)^d-\alpha^d x^z=f_\alpha(x).
\]

Conversely, let $A$ be a stochastic matrix with characteristic polynomial
$f_\alpha$.  By Lemma \ref{lem:typeII-exact-order} and the Ito parametrisation
in Section \ref{sec:background}, let
\[
        \lambda_\alpha\in K_n(\{q,s\})\setminus\Theta_{n-1}
\]
be the boundary root of $f_\alpha$.  Since $\chi_A=f_\alpha$, this
$\lambda_\alpha$ is an eigenvalue of $A$.  The hypotheses of the
Dmitriev--Dynkin/Kirkland--\v Smigoc reduction are therefore satisfied.
Applying that reduction and then the relabelling of Proposition
\ref{prop:blocknormal}, $A$ is permutation-similar to a weighted two-choice
matrix of the form \eqref{eq:matrixh}--\eqref{eq:matrixtlast}; the $d$
horizontal $q$-cycles have common weight $\beta$, and every remaining directed
cycle has length $s=qd-z$.  Therefore the transfer support is Type II
admissible in the sense of Definition \ref{def:admissible}.  Applying Theorem
\ref{thm:support} gives the interval presentation in (a).  The weight of the
horizontal $q$-cycle in block $t$ is $\prod_{i\in S_t}b_{t,i}$, so equality of
the $q$-cycle weights with $\beta$ gives \eqref{eq:weightproduct}.  This proves
the converse.
\end{proof}

\begin{corollary}[Irreducibility]
\label{cor:irreducible}
Every matrix in Theorem~\ref{thm:weighted} is irreducible.
\end{corollary}

\begin{proof}
The positive horizontal edges form a directed $q$-cycle in every block, so every vertex can reach every other vertex in the same block. Each transfer set $S_t$ is nonempty, and its transfer edges advance from block $t$ to block $t+1$ cyclically. Hence one can move from any block to any other block and then reach any prescribed vertex horizontally. The weighted digraph is strongly connected, which is equivalent to irreducibility of the nonnegative matrix.
\end{proof}

\begin{example}[$(q,d,z)=(5,2,2)$]
\label{ex:522}
Let $q=5$, $d=2$, and $z=2$. Then $n=10$, $s=8$, $c=4$, and
\[
        f_\alpha(x)=\bigl(x^5-\beta\bigr)^2-\alpha^2x^2,
        \qquad \beta=1-\alpha.
\]
Choose the composition $(a_0,a_1)=(1,1)$ and origin $\rho=0$. The intervals in Theorem~\ref{thm:support} are
\[
        I_0=\{0,1\},
        \qquad
        I_1=\{3,4\},
\]
and take the maximal supports $S_0=I_0$ and $S_1=I_1$. The corresponding weight cell has dimension
\[
        (|S_0|-1)+(|S_1|-1)=2=z,
\]
so it attains the non-sparsity bound of Section~\ref{sec:consequences}.

Set $\alpha=3/10$ and $\beta=7/10$, and choose
\[
        b_{0,0}=\frac78,
        \quad b_{0,1}=\frac45,
        \qquad
        b_{1,3}=\frac57,
        \quad b_{1,4}=\frac{49}{50}.
\]
Both block products equal $7/10$, as required by \eqref{eq:weightproduct}. Order the vertices as
\[
        (0,0),\ldots,(0,4),(1,0),\ldots,(1,4).
\]
Then the matrix \eqref{eq:matrixh}--\eqref{eq:matrixtlast} is
{\scriptsize
\[
A=\begin{pmatrix}
0&\tfrac78&0&0&0&\tfrac18&0&0&0&0\\
0&0&\tfrac45&0&0&0&\tfrac15&0&0&0\\
0&0&0&1&0&0&0&0&0&0\\
0&0&0&0&1&0&0&0&0&0\\
1&0&0&0&0&0&0&0&0&0\\
0&0&0&0&0&0&1&0&0&0\\
0&0&0&0&0&0&0&1&0&0\\
0&0&0&0&0&0&0&0&1&0\\
0&0&\tfrac27&0&0&0&0&0&0&\tfrac57\\
0&0&0&\tfrac1{50}&0&\tfrac{49}{50}&0&0&0&0
\end{pmatrix}.
\]
}
Direct expansion gives
\[
\begin{aligned}
\chi_A(x)
&=x^{10}-\frac75x^5-\frac9{100}x^2+\frac{49}{100}\\
&=\left(x^5-\frac7{10}\right)^2
  -\left(\frac3{10}\right)^2x^2
=f_{3/10}(x).
\end{aligned}
\]
The four transfer tuples in $S_0\times S_1$ give four simple $8$-cycles. For example, $(i_0,i_1)=(0,3)$ gives
\[
\begin{aligned}
(0,0)&\to(1,0)\to(1,1)\to(1,2)\to(1,3)\\
     &\to(0,2)\to(0,3)\to(0,4)\to(0,0).
\end{aligned}
\]

The interval restriction is sharp. Replace $S_1=\{3,4\}$ by
$S_1'=\{2,4\}$ and assign the same numerical weight values to the modified support positions. The resulting graph contains the forbidden $3$-cycle
\[
        (0,1)\to(1,1)\to(1,2)\to(0,1),
\]
and direct expansion gives
\[
        \chi_A(x)
        =x^{10}-\frac{2}{35}x^7-\frac75x^5
         -\frac{23}{700}x^2+\frac{49}{100}
        \ne f_{3/10}(x).
\]
The $x^7$ term records a directed cycle of length $10-7=3$, which is not among the permitted lengths $q=5$ and $s=8$.
\end{example}

\begin{figure}[t]
\centering
\includegraphics[width=\textwidth]{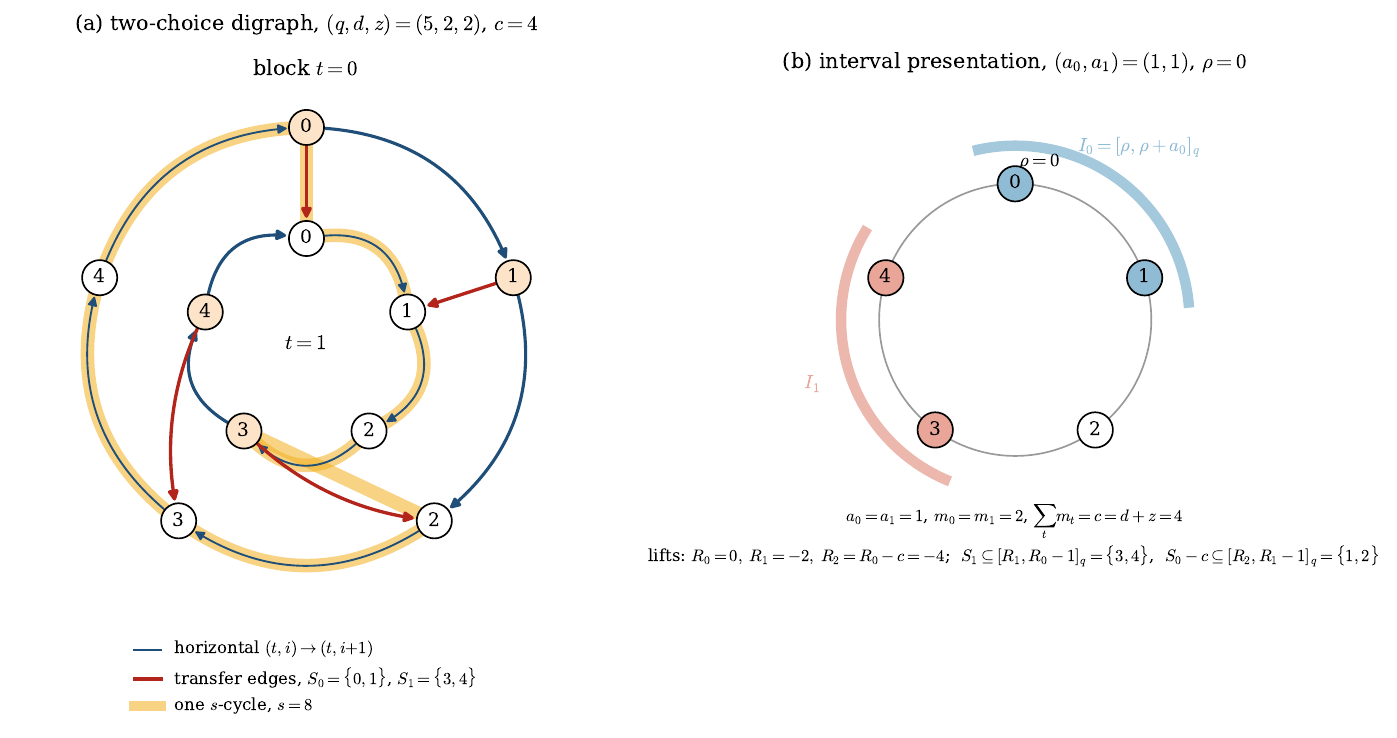}
\caption{The exact example in Example~\ref{ex:522}. (a) The two-choice digraph with $S_0=\{0,1\}$ and $S_1=\{3,4\}$; the shaded path is the transfer cycle associated with $(i_0,i_1)=(0,3)$. (b) The interval presentation for $(a_0,a_1)=(1,1)$ and $\rho=0$, showing the consecutive cyclic intervals underlying the common-blocker argument.}
\label{fig:digraph}
\end{figure}

\section{Consequences and scope}
\label{sec:consequences}

For a fixed admissible support family $S_0,\ldots,S_{d-1}$, the free weight parameters are exactly the numbers $b_{t,i}$ with $i\in S_t$, subject to one product constraint in each block.  Hence the dimension of the corresponding open weight cell is
\begin{equation}
        \sum_{t=0}^{d-1}(|S_t|-1).
        \label{eq:dimension}
\end{equation}
The sparsest realisers are precisely those for which $|S_t|=1$ for every $t$.  In that case every weight cell has dimension zero and the unique transfer edge in each block has weight $\alpha$.

The interval theorem also shows that the total amount of non-sparsity is bounded by $z$.  Indeed, if $S_t\subseteq I_t$ and $|I_t|=a_t+1$, then
\[
        \sum_{t=0}^{d-1}(|S_t|-1)
        \le
        \sum_{t=0}^{d-1}a_t
        =z .
\]
Thus $z$ is not merely the exponent of the exceptional term in the reduced Ito polynomial; it is the maximum global budget for additional transfer positions beyond the sparsest support, and the maximum is attained by taking all positions in the permitted intervals.

Kirkland and \v Smigoc parametrise the sparsest Type II case by integers $x_0,\ldots,x_{d-1}\in\{0,\ldots,q-1\}$ satisfying
\[
        x_0+\cdots+x_{d-1}=qd-z-d .
\]
Up to the cyclic reindexing used in Proposition \ref{prop:blocknormal},
the present parametrisation is the complementary one:
\[
        m_t=q-x_t,
        \qquad
        a_t=m_t-1=q-1-x_t .
\]
Then
\[
        \sum_{t=0}^{d-1}a_t
        =d(q-1)-\sum_{t=0}^{d-1}x_t
        =z .
\]
The non-sparse realisers are obtained by replacing the single transfer position in block $t$ by an arbitrary nonempty subset of the interval of length $a_t+1$ determined by the same cyclic carry data, and then distributing the horizontal weights in that block subject only to product $\beta$.

The parametrisation has the following precise scope. For the non-degenerate full-degree Type II reduced Ito polynomial
\[
        \bigl(x^q-(1-\alpha)\bigr)^d-\alpha^d x^z,
        \qquad n=qd,
        \qquad 1\le z\le q-1,
        \qquad \gcd(q,z)=1,
\]
every stochastic matrix with that characteristic polynomial is now described, up to permutation similarity, by explicit finite support data and elementary product constraints on the weights.

This leaves several adjacent problems untouched.
\begin{enumerate}[label=\textup{(\roman*)}]
\item The global analytic Karpelevich-arc problem remains separate. Lemma \ref{lem:typeII-exact-order} uses the radial boundary theorem only to compare the order-$(qd-1)$ and order-$qd$ arcs with the same Farey endpoints, thereby supplying the exact-order hypothesis for the Dmitriev--Dynkin/Kirkland--\v Smigoc normal form. The present parametrisation realises the polynomial, and hence all of its roots, as the spectrum of a stochastic matrix; it does not identify a global continuous root branch for every non-power arc, prove angular monotonicity, or prove extremality beyond the exact-order comparison used here.
\item The Type III geometry is not a corollary of the Type II support theorem. In Type III the normal form has a global $n$-cycle and backward $q$-edges; the companion manuscript treats that problem by different methods, using Coates' formula, weighted Tur\'an equality, and a circular telescoping argument \cite{VerbekenGinisTypeIII}.
\item The powers problem for non-sparse Type II realisers is open.  Joshi, Kirkland and \v Smigoc classify powers of Karpelevich arcs and powers of sparsest realising matrices \cite{JoshiKirklandSmigoc2023}; Theorem \ref{thm:weighted} gives the larger non-sparse class on which an analogous power/root question can now be asked.
\item The endpoint cases $\alpha=0$ and $\alpha=1$ are degenerate and are not included in Theorem \ref{thm:weighted}. Explicitly,
\[
        f_0(x)=(x^q-1)^d,
        \qquad
        f_1(x)=x^z\bigl(x^{qd-z}-1\bigr).
\]
The open-arc input used in the proof therefore fails, and additional reducible realisations are no longer governed by the same parametrisation.
\item The interval presentation in Theorem \ref{thm:support} is existential and nonunique.  A canonical enumeration of support classes modulo the natural cyclic symmetries would be a separate finite combinatorial refinement.
\item The broader problem of classifying all stochastic matrices having a specified boundary eigenvalue, while allowing additional spectral factors, is larger than the characteristic-polynomial realiser problem considered here.
\end{enumerate}

Thus Theorem~\ref{thm:weighted} gives a complete parametrisation, up to permutation similarity, for the full Type~II arithmetic range $q\ge2$, $d\ge2$, $1\le z\le q-1$, $\gcd(q,z)=1$, and $0<\alpha<1$. What remains is global analytic arc theory, endpoint structure, powers of non-sparse realisers, and canonical enumeration.

\section*{Funding}
Vincent Ginis acknowledges support from the Research Foundation -- Flanders (FWO) under grants No. G032822N and G0K9322N.

\section*{Declaration of competing interest}
The authors declare that they have no competing interests.

\section*{Data availability}
No data were used for the research described in this article.

\section*{Acknowledgements}
The authors would like to thank Carlo Emerencia for his enthusiastic and careful proofreading; his remarks made the manuscript substantially better.

\end{document}